\documentclass[oneside]{article} 
\usepackage{amsmath,amsthm,amssymb}     
\usepackage[mathscr]{eucal}     
\usepackage{hyperref}           
\usepackage{enumerate}          

\numberwithin{equation}{section}
\theoremstyle{definition}
\newtheorem{dfn}{Definition}[section]

\theoremstyle{plain}
\newtheorem{thm}[dfn]{Theorem}

\newtheorem{lem}[dfn]{Lemma}

\title{The isomorphism problem for multiparameter
quantized Weyl algebras\footnote{The research of the first-named author was supported by National Science Foundation grant DMS-0800948.} 
\footnote{2010 \emph{Mathematics Subject Classification}. Primary 16W35; secondary 16S36.}}
\author{K. R. Goodearl \and J. T. Hartwig}

 \newcommand{\ga}{\gamma}\newcommand{\Ga}{\Gamma}
\newcommand{\ep}{\varepsilon}
 \newcommand{\la}{\lambda}
  \newcommand{\si}{\sigma}   \newcommand{\ph}{\varphi}

\newcommand{\iv}[2]{[\![ #1,#2 ]\!]}  
\newcommand{\MQWA}[4]{A_{#1}^{#2,#3}(#4)}         

\newcommand\Z{\mathbb{Z}}

\newcommand\N{\mathbb{N}}\newcommand\K{\mathbb{K}}

\DeclareMathOperator{\Span}{Span}

\begin{document}

\maketitle

\begin{abstract}
In this note we solve the isomorphism problem for the  multiparameter quantized Weyl algebras, in the case when none of the deformation parameters $q_i$ is a root of unity, over an arbitrary field.
\end{abstract}

\section{Introduction}
The \emph{isomorphism problem} for a family of algebraic structures given by generators and relations, is the problem of determining whether or not two members in the family are isomorphic. This is a very hard problem in general; for the family of finitely presented groups, the problem is known to be undecidable \cite{A}, \cite{Ra}.

Quantum algebras are various noncommutative associative algebras arising in the theory of quantum groups. They have many interesting properties from the point of view of ring theory (see e.g. \cite{Smith1992}, \cite{Joseph1995}, \cite{BroGoo2002}, and references therein). When it comes to isomorphisms, the general rule is that quantum algebras are very rigid in the sense of having few automorphisms \cite{GomKao2002}, \cite{N}, \cite{Yakimov2012a}, \cite{Yakimov2012b}. This is in stark contrast with the commutative counterparts, where determining automorphism groups can be a wild problem \cite{SheUmi2003}.

In this paper we consider a family of quantum algebras, called \emph{multiparameter quantized Weyl algebras}, introduced in \cite{M}. They have been studied in many papers and are related to $q$-difference operators \cite{Jordan1995},\cite{M}, multiparameter quantum groups \cite{M}, twisted generalized Weyl algebras \cite{MazTur2002}, \cite{Hartwig2006}, \cite{FH} and iterated skew polynomial rings \cite{GooLet2000}, \cite{C}.
Some ring theoretic aspects that have been studied for these algebra include
global and Krull dimension  \cite{GiaZha1995}, \cite{FujKirKuz1999},
prime and primitive spectra \cite{R}, \cite{GooLet2000}, \cite{Goodearl2000},
localizations and division rings of fractions \cite{AD}, \cite{Jordan1995}. The automorphism group of a multiparameter quantized Weyl algebra was determined in \cite{R} for generic parameters.  The isomorphism problem for quantized Weyl algebras $A_1^q(\K)$ of degree $1$, with arbitrary parameters $q$, was solved in \cite{Gad}.

Isomorphism problems have been studied for certain classes of \emph{generalized Weyl algebras} \cite{Bavula}, \cite{RS}, \cite{BavJor2001}. Multiparameter quantized Weyl algebras are not however examples of generalized Weyl algebras, but rather of the more general class of twisted generalized Weyl algebras \cite{MazTur2002} for which no general results exist about isomorphisms or automorphisms.
In \cite{AD} some necessary conditions are given for two multiparameter quantized Weyl algebras to be isomorphic.

In the present paper we solve the isomorphism problem for multiparameter quantized Weyl algebras by determining necessary and sufficient conditions for two such algebras to be isomorphic, in the case when certain parameters $q_i$ are not roots of unity. The tools we use are Jordan's simple localization \cite{Jordan1995} and Rigal's methods \cite{R} for determining automorphisms.

\subsection*{Notation} 
Throughout, we work over an arbitrary base field $\K$. For any integers $a$ and $b$, let $\iv{a}{b}$ denote the set $\big\{x\in\Z\mid a\le x\le b\big\}$.

\section{Multiparameter quantized Weyl algebras}

\begin{dfn} \cite{M}
Let $n$ be a positive integer, $Q=(q_1,\ldots,q_n)$ an $n$-tuple in $(\K^\ast)^n$,
and $\Ga=(\gamma_{ij})$ a multiplicatively skew-symmetric $n\times n$ matrix with entries from $\K^\ast$.
The \emph{multiparameter quantized Weyl algebra} $\MQWA{n}{Q}{\Ga}{\K}$ is
defined as the unital associative $\K$-algebra with generators
\[x_1,y_1,x_2,y_2,\ldots,x_n,y_n\]
subject to defining relations
\begin{align}
\label{eq:rel1}
y_iy_j&=\ga_{ij}y_jy_i, &&\forall i,j\in\iv{1}{n}\\
\label{eq:rel2}
x_ix_j &= q_i\ga_{ij} x_jx_i,&&i<j\\
\label{eq:rel3}
x_iy_j &= \ga_{ji} y_jx_i,&&i<j\\
\label{eq:rel4}
x_iy_j &= q_j\ga_{ji} y_jx_i,&& i>j\\
\label{eq:rel5}
x_jy_j-q_jy_jx_j &= 1 + \sum_{k=1}^{j-1} (q_k-1)y_kx_k, &&\forall j\in\iv{1}{n}.
\end{align}
\end{dfn}

Frequently used are the elements
\begin{equation}
z_i:= [x_i,y_i]=1+\sum_{k=1}^i (q_k-1)y_kx_k,\qquad \forall i\in\iv{1}{n}.
\end{equation}
They are normal elements (e.g., \cite[\S2.8]{Jordan1995}) and satisfy the following relations ($z_0=1$ by convention):
\begin{equation}
x_iy_i-q_iy_ix_i = z_{i-1},\quad\forall i\in\iv{1}{n}.
\end{equation}
As is well known, $\MQWA{n}{Q}{\Ga}{\K}$ is an iterated skew polynomial algebra of the form
$$
\MQWA{n}{Q}{\Ga}{\K} =
\K[y_1][x_1; \tau_2,\delta_2] [y_2; \tau_3][x_2; \tau_4,\delta_4]
\cdots [y_n; \tau_{2n-1}][x_n; \tau_{2n},\delta_{2n}]
$$
(e.g., \cite[\SS2.1, 2.8]{Jordan1995}). Consequently, $\MQWA{n}{Q}{\Ga}{\K}$ is a noetherian domain. The Gelfand-Kirillov dimension of this algebra is $2n$ \cite[Proposition 3.4]{GLen}.

\section{Height one prime ideals}

We let $\langle Q,\Ga\rangle$ denote the subgroup of of $\K^\ast$
generated by the set
\[\big\{q_1,\ldots,q_n\big\}\cup\big\{\gamma_{ij}\mid i,j\in\iv{1}{n}, \, i<j\big\}.\]
We say that a multiparameter quantized Weyl algebra $\MQWA{n}{Q}{\Ga}{\K}$ is \emph{generic}
if $\langle Q,\Ga\rangle$ is free abelian of (the maximal possible) rank $n(n+1)/2$.

Rigal proved in \cite[Proposition 3.2.3]{R} that if $\K$ is algebraically closed of characteristic zero and $\MQWA{n}{Q}{\Ga}{\K}$ is generic, then the height one prime ideals of $\MQWA{n}{Q}{\Ga}{\K}$ are the principal ideals generated by the $z_i$. Using results of Jordan \cite{Jordan1995}, we can improve this result by weakening the assumptions on $\K$, $Q$ and $\Gamma$, as follows

\begin{thm} \label{height1}
Let $\MQWA{n}{Q}{\Ga}{\K}$ be a multiparameter quantized Weyl algebra. Assume that for each $i\in\iv{1}{n}$, $q_i$ is not a root of unity. Then the set of height one prime ideals of $\MQWA{n}{Q}{\Ga}{\K}$ equals
\[\big\{(z_1),(z_2),\ldots,(z_n)\big\}.\]
\end{thm}

\begin{proof} For each $i\in\iv{1}{n}$, the subalgebra of $A := \MQWA{n}{Q}{\Ga}{\K}$ generated by the elements $x_1,y_1,\dots, x_i,y_i$ is a multiparameter quantized Weyl algebra of the form $A_i := \MQWA{i}{Q^i}{\Ga^i}{\K}$, where $Q^i$ and $\Ga^i$ denote the restrictions of $Q$ and $\Ga$ to $\iv{1}{i}$ and $\iv{1}{i}^2$, respectively.
It follows from \cite[Proposition 2.7]{Jordan1995} that $z_i$ generates a completely prime ideal of $A_i$. In $A$, we have $(z_i) = z_i A$ by the normality of $z_i$, from which it follows that $A/(z_i)$ is an iterated skew polynomial ring over $A_i/ z_i A_i$, and thus is a domain. Consequently, $(z_i)$ is a completely prime ideal of $A$.

Next, let $Z$ be the multiplicative submonoid of $A$ generated by the normal elements $z_1,z_2,\ldots,z_n$. Jordan proved in \cite[Theorem 3.2]{Jordan1995} that the localization $A[Z^{-1}]$ is a simple ring. Consequently, any nonzero prime ideal $P$ of $A$ must meet $Z$, so $z_1^{k_1}z_2^{k_2}\cdots z_n^{k_n}\in P$ for some nonnegative integers $k_i$, not all equal to zero. Since $P$ is prime and $z_i$ is a normal element in $A$ for every $i$, it follows that $z_j \in P$ for some $j \in\iv{1}{n}$ and thus $(z_j) \subseteq P$. In particular, if $P$ has height one, then $P = (z_j)$.

Finally, since $(z_j) \not\subseteq (z_i)$ for any $i \ne j$, we conclude that no $(z_i)$ can properly contain a nonzero prime ideal, that is, $(z_i)$ has height one. (This also follows from the Noncommutative Principal Ideal Theorem \cite[Theorem 4.1.11]{McRob}.)
\end{proof}

\section{Automorphisms}

The automorphisms of $\MQWA{n}{Q}{\Ga}{\K}$ were determined by Rigal \cite[Th\'eor\`eme 4.2.5]{R} under the assumptions that $\K$ is algebraically closed of characteristic zero and $\MQWA{n}{Q}{\Ga}{\K}$ is generic. These hypotheses were mainly used in determining the height one prime ideals of $\MQWA{n}{Q}{\Ga}{\K}$. Theorem \ref{height1} can be used for this purpose, assuming only that the scalars $q_i$, for $i \in \iv{1}{n}$, are non-roots of unity. The remainder of Rigal's proof carries through as given in \cite{R}, yielding the following theorem.

\begin{thm}
\label{R4.2.5}
Let $\sigma$ be a $\K$-algebra automorphism of $\MQWA{n}{Q}{\Ga}{\K}$. If none of $q_1,\dots,q_n$ is a root of unity, there exists $(\mu_1,\dots,\mu_n) \in (\K^*)^n$ such that $\sigma(x_i) = \mu_i x_i$ and $\sigma(y_i) = \mu_i^{-1} y_i$ for all $i \in \iv{1}{n}$.
\end{thm}

We shall also need to upgrade one of the lemmas used in the proof of the above theorem, \rm\cite[Lemme~4.2.2]{R}, under the same weakened hypotheses.

\begin{lem}
\label{R4.2.2}
Assume that none of $q_1,\dots,q_n$ is a root of unity.
Give $\MQWA{n}{Q}{\Ga}{\K}$ an $\N$-filtration by specifying total degrees
\begin{align*}
d(x_i) &=d(y_i)= 0,&&\forall i\in\iv{1}{n-1},\\
d(x_n) &=d(y_n)= 1.&&
\end{align*}
If $a,b\in A_n^{Q,\Ga}(\K)\setminus\K$ such that
$d(ab)=2$ and $ab\in\Span_\K\{z_1,\ldots,z_n\}$, then
either $(a,b)=(\mu x_n, \nu y_n)$ or $(a,b)=(\mu y_n, \nu x_n)$ for some $\mu,\nu\in\K^\ast$.
\end{lem}

\section{Solution to the isomorphism problem}

\begin{thm}
Let $\MQWA{n}{Q}{\Ga}{\K}$ and $\MQWA{m}{Q'}{\Ga'}{\K}$ be two multiparameter quantized Weyl algebras, with standard generators $x_i$, $y_i$ and $x'_i$, $y'_i$, respectively.
\begin{enumerate}[{\rm (I)}]
\item Assume that  $q_1,\ldots,q_n$ and $q_1',\ldots,q_m'$ are not roots of unity. Then 
$\MQWA{n}{Q}{\Ga}{\K}$ and $\MQWA{m}{Q'}{\Ga'}{\K}$ are isomorphic as $\K$-algebras if and only if the following two conditions hold:
\begin{enumerate}[{\rm (i)}]
\item $m=n$;
\item There exists a sign vector $\ep\in\{1,-1\}^n$ such that
for all $i\in\iv{1}{n}$,
\begin{equation} \label{eq:Q'_cond}
 q_i'= q_i^{\ep_i};
\end{equation}
and for all $i,j\in\iv{1}{n}$ with $i<j$,
\begin{equation} \label{eq:Ga'_cond}
\gamma'_{ij} =
\begin{cases}
\gamma_{ij}, & \text{if $(\ep_i,\ep_j)=(1,1)$},\\
\gamma_{ji}, & \text{if $(\ep_i,\ep_j)=(-1,1)$},\\
q_i^{-1}\gamma_{ji}, & \text{if $(\ep_i,\ep_j)=(1,-1)$},\\
q_i\gamma_{ij}, & \text{if $(\ep_i,\ep_j)=(-1,-1)$}.
\end{cases}
\end{equation}
\end{enumerate}
\item If conditions {\rm(i)} and {\rm(ii)} above hold, then for any $\mu\in(\K^\ast)^n$
and $\ep \in \{\pm1\}^n$, there exists a unique $\K$-algebra isomorphism 
\[\ph_{\mu,\ep} :\MQWA{n}{Q}{\Ga}{\K}\longrightarrow\MQWA{m}{Q'}{\Ga'}{\K}\]
such that
\begin{align*}
\ph_{\mu,\ep}(x_i) &= \begin{cases} \mu_i x_i',& \ep_i=1\\ \mu_i y_i',& \ep_i=-1\end{cases}\\
\ph_{\mu,\ep}(y_i) &= \begin{cases} \la_i\mu_i^{-1} y_i',& \ep_i=1\\ -\la_i\mu_i^{-1} x_i',& \ep_i=-1\end{cases}
\end{align*}
for all $i \in \iv{1}{n}$,
where $\la=\la(\ep)\in(\K^\ast)^n$ is defined by the recursion relation
\begin{equation*}
\la_i = q^{(\ep_i-1)/2}\la_{i-1},\quad \la_0=1.
\end{equation*}
Finally, $(\mu,\ep) \longmapsto \ph_{\mu,\ep}$ defines a bijection between $(\K^\ast)^n \times \{\pm1\}^n$ and
the set of isomorphisms
$\MQWA{n}{Q}{\Ga}{\K}\longrightarrow\MQWA{m}{Q'}{\Ga'}{\K}$.
\end{enumerate}
\end{thm}

\begin{proof} Set $A := \MQWA{n}{Q}{\Ga}{\K}$ and $A' := \MQWA{m}{Q'}{\Ga'}{\K}$, and set $z'_i := [x'_i,y'_i] \in A'$ for $i \in \iv{1}{m}$, as well as $z'_0 = 1 \in A'$. For $i \in \iv{1}{n}$, write $A_i$ for the subalgebra of $A$ generated by $x_1,y_1, \dots, x_i,y_i$, and write $A'_j$ for the analogous subalgebras of $A'$.

(I): Suppose $\varphi: A\to A'$ is an isomorphism of $\K$-algebras.
Since the Gelfand-Kirillov dimensions of $A$ and $A'$ are $2n$ and $2m$, it is immediate that $m=n$.
Now $\varphi$ must map the set of height one prime ideals of $A$ bijectively onto the corresponding set for $A'$. In view of Theorem \ref{height1}, we conclude that there
exist $\la\in(\K^\ast)^n$ and a permutation $\si\in S_n$ such that
\begin{equation*}
\varphi(z_i)=\la_i z_{\si(i)}',\quad\forall i\in\iv{1}{n}.
\end{equation*}

Give $A'$ the $\N$-filtration specified in Lemma \ref{R4.2.2}.
Let $i=\si^{-1}(n)$. By definition of $z_i$ we have $z_i-z_{i-1}=(q_i-1)y_ix_i$. Applying $\varphi$ to both sides gives
\[ \la z_n' - \varphi(z_{i-1}) = (q_i-1)\varphi(y_i)\varphi(x_i), \]
which implies
\[ \la (q_n'-1)y_n'x_n' + c = (q_i-1)\varphi(y_i)\varphi(x_i) \]
where $d(c)=0$, and so $d( \varphi(y_i)\varphi(x_i) ) = 2$. Now we use Lemma \ref{R4.2.2}.
We obtain that for $i=\si^{-1}(n)$, 
\begin{equation}\label{eq:yixi1}
\big(\varphi(x_i),\varphi(y_i)\big)\in\big\{
(\mu x_n', \nu y_n'), \, (\mu y_n', \nu x_n')\mid \mu,\nu\in\K^\ast \big\}.
\end{equation}

Suppose $i<n$. Then we can apply the same argument as above
to the relation $z_{i+1}-z_i=(q_{i+1}-1)y_{i+1}x_{i+1}$
and obtain
\begin{equation}\label{eq:yixi2}
\big(\varphi(x_{i+1}),\varphi(y_{i+1})\big)\in\big\{
(\mu x_n', \nu y_n'), \, (\mu y_n', \nu x_n')\mid \mu,\nu\in\K^\ast \big\}.
\end{equation}
But then \eqref{eq:yixi1} and \eqref{eq:yixi2} imply that $x_n'$ has two different pre-images under $\varphi$,
which is obviously impossible. Therefore $i=n$. So
$\varphi(z_n)=\la_n z_n'$ and
\begin{equation}\label{eq:xnyn}
\big(\varphi(x_n),\varphi(y_n)\big)\in\big\{
(\mu x_n', \nu y_n'), \, (\mu y_n', \nu x_n')\mid \mu,\nu\in\K^\ast \big\}.
\end{equation}
In particular, for $j<n$ we have $d(\varphi(y_j)\varphi(x_j))=d(\varphi(z_j) - \varphi(z_{j-1}))=0$ which shows that $\varphi(x_j),\varphi(y_j)\in A'_{n-1}$.
So $\varphi (A_{n-1})\subseteq A'_{n-1}$. Hence, applying the argument
leading to \eqref{eq:xnyn} to $\varphi|_{A_{j}}$ for $j=n-1,n-2,\ldots,1$ we obtain
\begin{gather*}
\varphi(z_i)=\la_i z_i',\quad\forall i\in\iv{1}{n},\\
\label{eq:xiyi}
\big(\varphi(x_i),\varphi(y_i)\big)\in\big\{
(\mu_i x_i', \nu_i y_i'), \, (\mu_i y_i', \nu_i x_i')\mid \mu_i,\nu_i\in\K^\ast \big\},\quad\forall i\in\iv{1}{n}.
\end{gather*}

We define $\ep\in\{1,-1\}^n$ by
\begin{equation*}\label{eq:ep_def}
\ep_i =
\begin{cases}
1,&\text{if $(\varphi(x_i),\varphi(y_i))=(\mu_i x_i', \nu_i y_i')$ for some $\mu_i,\nu_i\in\K^\ast$,}\\
-1,&\text{if $(\varphi(x_i),\varphi(y_i))=(\mu_i y_i', \nu_i x_i')$ for some $\mu_i,\nu_i\in\K^\ast$.}
\end{cases}
\end{equation*}
Next we prove that \eqref{eq:Q'_cond} holds. Applying $\varphi$ to the relation
$x_i y_i - q_i y_i x_i = z_{i-1}$ we obtain
\begin{align*}
\mu_i \nu_i (x_i' y_i'- q_i y_i'x_i') &= \la_{i-1} z_{i-1}', \qquad \ep_i =1, \\
\mu_i \nu_i (y_i' x_i'- q_i x_i'y_i') &= \la_{i-1} z_{i-1}', \qquad \ep_i =-1.
\end{align*}
Using $x_i' y_i'- q_i' y_i' x_i' = z_{i-1}'$, we get
\begin{gather*}
\mu_i \nu_i (q_i' - q_i) y_i'x_i'   \in A'_{i-1}, \qquad\ep_i=1, \\
\mu_i \nu_i ((q_i')^{-1}- q_i) x_i'y_i' \in A'_{i-1}, \qquad\ep_i=-1.
\end{gather*}
Statement \eqref{eq:Q'_cond} follows immediately.

Let $i,j\in\iv{1}{n}$. Assume that $i<j$.
We apply $\varphi$ to relation \eqref{eq:rel1} to obtain
\begin{align*} 0 &= (\nu_i\nu_j)^{-1}\big(\varphi(y_i)\varphi(y_j)-\ga_{ij}\varphi(y_j)\varphi(y_i)\big) \\
 &= \begin{cases}
y_i'y_j'-\ga_{ij} y_j'y_i'=(\ga_{ij}'-\ga_{ij})y_j'y_i'& (\ep_i,\ep_j)=(1,1)\\
x_i'y_j'-\ga_{ij} y_j'x_i'=(\ga_{ji}'-\ga_{ij})y_j'x_i'& (\ep_i,\ep_j)=(-1,1)\\
y_i'x_j'-\ga_{ij} x_j'y_i'=(1-\ga_{ij}q_i'\ga_{ij}')y_i'x_j' & (\ep_i,\ep_j)=(1,-1)\\
x_i'x_j'-\ga_{ij} x_j'x_i'=(q_i'\ga_{ij}'-\ga_{ij} )x_j'x_i' & (\ep_i,\ep_j)=(-1,-1).
\end{cases}
\end{align*}
Combining this with \eqref{eq:Q'_cond}, we obtain
\eqref{eq:Ga'_cond}. This proves the necessity of
conditions (i), (ii).

Conversely, assume that (i), (ii) hold. Define elements
\begin{equation*}
x_i^* = \begin{cases} x_i',& \ep_i=1\\ y_i',& \ep_i=-1\end{cases} \qquad\text{and}\qquad
y_i^* = \begin{cases} \la_i y_i',& \ep_i=1\\ -\la_i x_i',& \ep_i=-1\end{cases}
\end{equation*}
in $A$ for $i \in \iv{1}{n}$, 
where $\la=\la(\ep)\in(\K^\ast)^n$ is defined by the recursion relation
\begin{equation*}\label{eq:lambda_rec}
\la_i = q_i^{(\ep_i-1)/2}\la_{i-1},\quad \la_0=1.
\end{equation*}
It is straightforward to verify that these elements satisfy the relations \eqref{eq:rel1}-\eqref{eq:rel4}. As for \eqref{eq:rel5}, note that this is equivalent to
\begin{gather}
\label{2.5.1}
x^*_1 y^*_1 - q_1 y^*_1 x^*_1 = 1,  \\
\begin{gathered}
\label{2.5.j}
(x^*_j y^*_j - q_j y^*_j x^*_j) - (x^*_{j-1} y^*_{j-1} - q_{j-1} y^*_{j-1} x^*_{j-1}) = (q_{j-1}-1) y^*_{j-1} x^*_{j-1}, \\ \forall j \in \iv{2}{n},
\end{gathered}
\end{gather}
and that \eqref{2.5.j} is equivalent to $x^*_j y^*_j - q_j y^*_j x^*_j = [x^*_{j-1}, y^*_{j-1}]$. For both choices of $\ep_{j-1}$, the commutator $[x^*_{j-1}, y^*_{j-1}]$ equals $\la_{j-1} [x'_{j-1}, y'_{j-1}]$. Consequently, \eqref{2.5.1} and \eqref{2.5.j} will both follow from
\begin{equation}
\label{2.5.1j}
x^*_j y^*_j - q_j y^*_j x^*_j = \la_{j-1} z'_{j-1}, \qquad \forall j \in \iv{1}{n}.
\end{equation}
It is straightforward to verify \eqref{2.5.1j}.

Since the elements $x^*_i$, $y^*_i$ satisfy the relations \eqref{eq:rel1}-\eqref{eq:rel5}, there is a unique $\K$-algebra homomorphism $\varphi: A \to A'$ such that
$$
\varphi(x_i) = x^*_i \qquad\text{and}\qquad \varphi(y_i) = y^*_i, \qquad \forall i\in \iv{1}{n}.
$$
It is clearly surjective, since its image contains $x'_i$ and $y'_i$ for all $i \in \iv{1}{n}$. Injectivity follows because $A$ and $A'$ are domains with the same finite Gelfand-Kirillov dimension (see \cite[Proposition 3.15]{KL}).

(II): This follows from the proof of part (I) and Theorem \ref{R4.2.5}.
\end{proof}

\noindent Department of Mathematics, University of California, Santa Barbara, CA 93106, USA

\noindent goodearl@math.ucsb.edu
\smallskip

\noindent Department of Mathematics, University of California, Riverside, CA 92521, USA

\noindent hartwig@math.ucr.edu

\end{document}